\documentclass[12pt]{amsart}
\usepackage{amssymb}
\usepackage{amsfonts}
\usepackage{amssymb,amstext,amscd,amsmath,mathdots}
\usepackage{fullpage}
\usepackage[all]{xy}
 \usepackage[usenames,dvipsnames]{color}

\usepackage{graphicx}
\usepackage{bbding}
\usepackage[misc]{ifsym}

\newenvironment{thm}{\subsection{}{\textbf {Theorem.}}\em}{}

\newenvironment{cor}{\subsection{}{\textbf {Corollary.}}\em}{}
\newenvironment{lem}{\subsection{}{\textbf {Lemma.}}\em}{}

\newenvironment{defn}{\subsection{}{\textbf {Definition.}}\em}{\smallskip}
\newenvironment{eg}{\subsection{}{\textbf {Example.}}}{\smallskip}

\newenvironment{rem}{\subsection{}{\textbf {Remark.}}}{\smallskip}

\newcommand{\bC}{{\mathbb{C}}}

\newcommand{\bM}{{\mathbb{M}}}
\newcommand{\bN}{{\mathbb{N}}}

\newcommand{\bR}{{\mathbb{R}}}

  \newcommand{\A}{{\mathcal{A}}}
  \newcommand{\B}{{\mathcal{B}}}
  \newcommand{\C}{{\mathcal{C}}}

\renewcommand{\H}{{\mathcal{H}}}

  \newcommand{\K}{{\mathcal{K}}}

\renewcommand{\P}{{\mathcal{P}}}

\renewcommand{\phi}{\varphi}
\newcommand{\upchi}{{\raise.35ex\hbox{$\chi$}}}

\newcommand{\fA}{{\mathfrak{A}}}
\newcommand{\fB}{{\mathfrak{B}}}

\newcommand{\ran}{\operatorname{Ran}}
\newcommand{\rank}{\operatorname{rank}}

\newcommand{\tr}{\textup{tr}}

\newcommand{\beq}{\begin{eqnarray}}
\newcommand{\eneq}{\end{eqnarray}}

\begin{document}

\title{Operators which are polynomially isometric\\ to a normal operator}

\author[L.W. Marcoux]{Laurent W. Marcoux}
\address{Laurent W. Marcoux, Department of
Pure Mathematics\\ University of Waterloo\\ Waterloo, Ontario N2L 3G1\ CANADA}
\email{Laurent.Marcoux@uwaterloo.ca}

\author[Y. H. Zhang]{Yuanhang Zhang}
\address{Yuanhang Zhang, School of Mathematics\\Jilin University\\Changchun 130012\\P.R. CHINA}
\email{zhangyuanhang@jlu.edu.cn}

\thanks{L.W. Marcoux's research is supported in part by NSERC (Canada); Y.H. Zhang's research is supported by the Natural Science Foundation
for Young Scientists of Jilin Province (No.: 20190103028JH), 
 NNSF of China (No.: 11601104, 11671167, 11201171) and the China Scholarship Council.}

\begin{abstract}
Let $\H$ be a complex, separable Hilbert space and $\B(\H)$ denote the algebra of all bounded linear operators acting on $\H$.  Given a unitarily-invariant norm $\| \cdot \|_u$ on $\B(\H)$ and two linear operators $A$ and $B$ in $\B(\H)$, we shall say that $A$ and $B$ are \emph{polynomially isometric relative to} $\| \cdot \|_u$ if  $\| p(A) \|_u = \| p(B) \|_u$ for all polynomials $p$.  In this paper, we examine to what extent an operator $A$ being polynomially isometric to a normal operator $N$ implies that $A$ is itself normal.   More explicitly, we first show that if $\| \cdot \|_u$ is any unitarily-invariant norm on $\mathbb{M}_n(\bC)$, if $A, N \in \mathbb{M}_n(\mathbb{C})$ are polynomially isometric and $N$ is normal, then $A$ is normal.  We then extend this result to the infinite-dimensional setting by showing that if $A, N \in \B(\H)$ are polynomially isometric relative to the operator norm and $N$ is a normal operator whose spectrum neither disconnects the plane nor has interior, then $A$ is normal, while if the spectrum of $N$ is not of this form, then there always exists a non-normal operator $B$ such that $B$ and $N$ are polynomially isometric.  Finally, we show that if $A$ and $N$ are compact operators with $N$ normal, and if $A$ and $N$ are polynomially isometric with respect to the $(c,p)$-norm  studied by Chan, Li and Tu, then $A$ is again normal.

\end{abstract}

\date{\today}
\maketitle


\section{Introduction}
Let $\H$ be a complex, separable Hilbert space and denote by $\B(\H)$ the algebra of bounded linear operators acting on $\H$.  When $\H$ is finite-dimensional with $\mathrm{dim}\, \H = n < \infty$, we identify $\B(\H)$ with the algebra $\bM_n(\bC)$ of $n \times n$ complex matrices.   We denote by $\sigma(A)$ the spectrum of $A \in \B(\H)$.

Amongst the most studied and best understood class of operators in $\B(\H)$ is the set  of normal operators (i.e. those operators $N \in \B(\H)$ for which $N N^* = N^* N$).   There has been a great deal of work done to produce alternate conditions on a given $T \in \B(\H)$ which are either equivalent to $T$ being normal, or at least imply that $T$ is normal~\cite{BC2018, Ger2012}.   For example, it is easy to see that if $T \in \B(\H)$ and $\| T x \| = \| T^* x \|$ for all $x \in \H$, then $T$ is normal.  The articles~\cite{EI1998, GJSW1987} provide no fewer than $89$ such conditions in the matrix setting.

The present article adds yet another condition to this seemingly endless list.   Recall that a norm $\| \cdot \|_u$ on $\B(\H)$ is said to be \textbf{unitarily-invariant} if 
\[
\| X \|_u = \| U X \|_u = \| X V \|_u \]
for all $X \in \B(\H)$ and for all unitary operators $U, V \in \B(\H)$.   For example, the Frobenius and operator (i.e. spectral) norms are both unitarily-invariant norms on $\bM_n(\bC)$.

Given an operator $T \in \B(\H)$,  and given a unitarily-invariant norm $\| \cdot \|_u$ on $\B(\H)$, we are interested in determining to what extent the values of $\| p(T)\|_u$, $p$ a polynomial determine the normality of $T$.   To make this notion precise, we introduce the following definition.

%
%

\begin{defn} \label{sec1.02}
Let $A, B \in \B(\H)$,  and let $\|\cdot\|_u$ be a unitarily-invariant norm on $\B(\H)$.  We say that $A$ and $B$ are \textbf{polynomially isometric relative to} $\| \cdot \|_u$ if 
\[
\| p(A)\|_u = \| p(B) \|_u \mbox{ \ \ \ \ \ for all polynomials } p. \]
\end{defn}

We remark that when $\H$ is finite-dimensional with $\dim \H = n <\infty$, it follows from  the Cayley-Hamilton Theorem that one need only consider polynomials of degree at most $n-1$.  We should mention that variants of this notion have at times been referred to by saying that ``$A$ and $B$ have the same norm behaviour".  We feel that the current terminology is more precise.

%
%

\subsection{} \label{sec1.03}
The underlying theme of our investigation is that if $A$ and $B$ are linear operators which are polynomially isometric with respect to a given norm, then $A$ and $B$  share many common properties.   For example, in the setting of matrices, it is easily seen that if $A, B \in \bM_n(\bC)$ are polynomially isometric relative to a unitarily-invariant norm $\| \cdot\|_u$, then  $A$ and $B$ share a common minimal polynomial, and in particular, $\sigma(A) = \sigma(B)$.
In the case of the Frobenius and operator (i.e. spectral) norms, Greenbaum and Trefethen~\cite{GT1993}  studied the relationship between polynomial isometry and resolvent growth for a pair $A, B \in \bM_n(\bC)$, showing in particular that if $A$ and $B$ are polynomially isometric relative to any matrix norm, then $A$ and $B$ share common pseudospectrum.  Later, Viswanath and Trefethen~\cite{VisTre96} showed that if $A \in \bM_n(\bC)$ is unitarily equivalent to $A_1 \oplus A_2$ where $\sigma(A_1) \cap \sigma(A_2)= \varnothing$, and if $B \in \bM_n(\bC)$ is polynomially isometric to $A$, then $B$ also admits a direct sum decomposition $B = B_1 \oplus B_2$.
As yet another example, in~\cite[Theorem 2.1]{FFGSS11}, \cite[Theorem 2.1]{FGS11}, it was shown that if $A \in \bM_n(\bC)$ is an upper triangular Toeplitz matrix with nonzero superdiagonal and  $B \in \bM_n(\bC)$ is polynomially isometric to $A$ with respect to either the Frobenius or operator norm, then $A$ and $B$ are unitarily equivalent.

The main question we study is the following.

\bigskip

\noindent{\bf Question:}   Let $A, N \in \B(\H)$ be two operators with $N$ normal, and suppose that $\B(\H)$ is equipped with a unitarily-invariant norm $\| \cdot \|_u$.   If $A$ and $N$ are polynomially isometric with respect to $\| \cdot \|_u$, must $A$ be normal?

\bigskip

In the matrix setting, some positive answers have already been obtained.  For $\bM_n(\bC)$  equipped with the Frobenius norm, a positive answer to this question was given in \cite{Ger2012},  while the case of $\bM_n(\bC)$ equipped with the usual operator norm was subsequently handled in \cite{BC2018}. 

In Section 2 below, we show that for \emph{any} unitarily-invariant norm $\|\cdot\|_u$ on $\bM_n(\bC)$, if $A \in \bM_n(\bC)$ is polynomially isometric to a normal matrix $N$, then $A$ is normal.  Furthermore, if the norm is  able to distinguish between projections of different rank, then $A$ and $N$ must in fact be unitarily equivalent.

From Section~\ref{sec03} on, we assume that  $\H$ is an infinite-dimensional, complex, separable Hilbert space.
The most important unitarily-invariant norm in the infinite-dimensional setting is the usual operator norm $\| \cdot \|$ on $\B(\H)$.   In Section~\ref{sec03} we consider the case where $A, N \in \B(\H)$ are polynomially isometric relative to this norm, and $N$ is normal.   A compact subset $K \subseteq \bC$ is said to be \textbf{Lavrientiev} if $\bC \setminus K$ is connected, and if the interior $K^o$ of $K$ is empty.  We prove that if $\sigma(N)$ is Lavrentieff, then $A$ is normal, while if $\sigma(N)$ is not Lavrentieff, then there always exists a non-normal operator $B$ such that $B$ and $N$ are polynomially isometric relative to the operator norm.

In Section~\ref{sec04}, we extend above result to compact operators $A$ and $N$ (with $N$ normal) which are polynomially isometric with respect to the $(c,p)$-norms introduced by Chan, Li and Tu in~\cite{CLT2001}.   Since the operator norm on $\B(\H)$ is an example of a $(c,p)$-norm, we see from above that the result does not extend to all pairs of operators $A$ and $N$ in $\B(\H)$ with $N$ normal.

We conclude with a few remarks concerning polynomial isometry when restricted to polynomials without constant term.


\section{Unitarly-invariant norms on $\bM_n(\bC)$}

%
%

We first deal with the case where $\H$ is finite-dimensional.  The reader is referred to \cite{Fan51,Mir60,LiTsi87} for more information about the relationship between unitarily-invariant norms and symmetric gauge functions.

%
%

\begin{defn}
Let $x=(x_1,\cdots,x_n), y=(y_1,\cdots, y_n)\in \bR^n$.  We shall write  $x\leq y$ to mean that $x_i\leq y_i(i=1,\cdots,n).$ 
Let us also denote by $\tilde{x}_1\,\cdots, \tilde{x}_n$ the numbers $x_1,\cdots, x_n$ arranged in \emph{non-ascending}
order of magnitude, and let $\tilde{y}_1\cdots, \tilde{y}_n$ be defined analogously. If
\[\tilde{x}_1+\cdots+\tilde{x}_k\leq \tilde{y}_1+\cdots+\tilde{y}_k\ \ \ \ \ \ \ (k=1,...,n), \]
we shall write $x\preccurlyeq y$.
\end{defn}
%
%

Next, we need to recall the following fact which will be used later.

\begin{lem}\label{comparison}
Let $p,q,k\in \bN$, and $0<r_1<\cdots<r_k$ be positive real numbers.
Set
\[x=(\overbrace{0,\cdots,0}^{p};\overbrace{1,\cdots,1}^{q}; \overbrace{1,\cdots,1}^{k}),~~~~ y=(\overbrace{0,\cdots,0}^{p};\overbrace{1,\cdots,1}^{q}; 1+r_1, \cdots, 1+r_k),\]
Then for every symmetric gauge function $\Phi$ in $p+q+k$ variables, we have that 
\[\Phi(x)<\Phi(y).\]
\end{lem}

\begin{proof}
Set
\[\hat{y}=(\overbrace{0,\cdots,0}^{p};\overbrace{1+\frac{r_1}{2q},\cdots,1+\frac{r_1}{2q}}^{q}; 1+\frac{r_1}{2}, ,1+r_2, \cdots, 1+r_k).\]
Then
\[\hat{y}\preccurlyeq y,\]
and
\[\hat{y}\geq z:= (\overbrace{0,\cdots,0}^{p};\overbrace{1+\frac{r_1}{2q},\cdots,1+\frac{r_1}{2q}}^{q}; \overbrace{1+\frac{r_1}{2q},\cdots,1+\frac{r_1}{2q}}^{k}).\]
Therefore, by~\cite[Theorem 4]{Fan51} (also see~\cite[Theorem 1]{Mir60}), for a fixed symmetric gauge function $\Phi$ in  $p+q+k$ variables,
\[\Phi(y)\geq \Phi(\hat{y})\geq \Phi(z)=(1+\frac{r_1}{2q})\Phi(x)>\Phi(x).\]
\end{proof}

%
%

For convenience, we formulate a variant of \cite[Proposition 1.3]{VisTre96} as the next lemma, which will prove very useful.

\begin{lem}\label{projection}
Let $P\in \bM_n(\bC)$ be a projection, $Q\in \bM_n(\bC)$ be an idempotent, and $\|\cdot\|_u$ be a unitarily-invariant norm of $\bM_n(\bC)$. If $\|P\|_u=\|Q\|_u$ and $\|I_n-P\|_u=\|I_n-Q\|_u$, then $Q$ is a projection.
\end{lem}

\begin{proof}
Suppose that $\rank(P)=r$, $\rank(Q)=s$.
If $s=0$, then $Q=0$, while
if $s=n$, then $\rank(I_n-Q)=0$, that is, $Q=I_n$.

We therefore assume that $1\leq s\leq n-1$.
Relative to the decomposition $\bC^n = \mathrm{ran}\, Q \oplus (\mathrm{ran}\, Q)^\perp$, we may write
\[Q=
\begin{bmatrix}
I_s& Y\\
0&0_{n-s}
\end{bmatrix},
\]
Note that either $s\geq r$ or $n-s > n-r$.
\begin{enumerate}
\item When $s\geq r$, we note that since the singular values of $Q$ are same (counting multiplicity) as the singular values of $Q^*$,
it follows that $\|Q^*\|_u=\|Q\|_u$. If $Y\neq 0$, then -- by Lemma~\ref{comparison} and the one-to-one  correspondence  between unitarily-invariant norms and symmetric gauge functions due to Von Neumann (see~\cite[Theorem 1.1]{LiTsi87}) --  $\|Q\|_u>\|P\|_u$.

\item When $n-s > n-r$, we consider
\[I_n-Q=
\begin{bmatrix}
0& -Y\\
0&I_{n-s}
\end{bmatrix}.\]

In an argument similar to the one above, $\|I_n-Q\|_u>\|I_n-P\|_u$ by Lemma~\ref{comparison} unless $Y = 0$.
\end{enumerate}
Therefore, $Y=0$, and so $Q$ is a projection.
\end{proof}

%
%

Suppose that $P$ and $Q \in \bM_n(\bC)$ are two projections of equal rank.  Clearly $P$ and $Q$ are unitarily similar, and thus if $\| \cdot\|_u$ is any unitarily-invariant norm on $\bM_n(\bC)$, then $\| P \|_u = \| Q\|_u$.   In general, however, it often happens that $\| P\|_u = \| Q\|_u$ even if the ranks of $P$ and $Q$ differ;  one need only consider the operator norm on $\bM_n(\bC)$.   Let us agree to say that a unitarily-invariant norm $\| \cdot \|_u$ \textbf{separates projections by rank} if, whenever $P$ and $Q$ are projections and $\mathrm{rank}\, P \ne \mathrm{rank}\, Q$,  then $\| P \|_u \ne \| Q \|_u$.   This happens for a large family of unitarily-invariant norms, including, for example, the $p$-norms 
\[
\| T \|_p := \left( \sum_{k=1}^n s_k(T)^p \right)^{\frac{1}{p}}, \ \ \ \ T \in \bM_n(\bC), \]
where $s_1(T) \ge s_2(T) \ge \cdots s_n(T) \ge 0$ are the singular numbers of $T$.  (See Section~\ref{sec04} for the definition of the singular numbers of $T \in \B(\H)$.)

The following theorem generalises the equivalence of conditions (i) and (vii) in the main Theorem of~\cite{Ger2012} as well as~\cite[Corollary~1]{BC2018}, and it adds yet one more equivalence to the 89 other characterisations
of normal matrices that appear in~\cite{GJSW1987} and~\cite{EI1998}.

\begin{thm} \label{thm2.04}
Let $n\in \bN$ and $\|\cdot\|_u$ be a unitarily-invariant norm on
$\bM_n(\bC)$. Let $N$ be an $n\times n$ normal complex matrix and $A\in \bM_n(\bC)$.
Suppose that $A$ and $N$ are polynomially isometric relative to $\| \cdot \|_u$.  
Then $A$ is normal.

If, furthermore, we suppose that $\| \cdot \|_u$ separates projections by rank, then $A$ and $N$ are unitarily similar.
\end{thm}

\begin{proof}
Let $\mu_N$ be the minimal polynomial of $N$.   Then $\sigma(N)=\{\lambda: \mu_N(\lambda)=0\}$.
Since $\mu_N(N)=0$, it follows that $\mu_N(A)=0$, and thus 
\[\mu_N(\sigma(A))=\sigma(\mu_N(A))=\{0\}.\]
In particular, $\sigma(A)\subset \sigma(N)$. By symmetry,
\[\sigma(N)=\sigma(A):= \{\lambda_1, \cdots, \lambda_m\}.\]

If $m=1$, then $N=\lambda I$, and since
\[\|A-\lambda I\|=\|N-\lambda I\|=0,\]
we have $A=\lambda I$.

Next, suppose that $m\geq 2$.
Define
\[p_i(z)=\prod_{1\leq j\leq m, j\neq i}\frac{z-\lambda_j}{\lambda_i-\lambda_j},~~1\leq i\leq m.\]
Set
\[P_i:= p_i(N),~~Q_i:= p_i(A),\]
it follows that
\[\|P_i\|_u=\|Q_i\|_u, \mbox{ and } ~~\|I_n-P_i\|_u=\|I_n-Q_i\|_u.\]

From the basic properties of the continuous functional calculus for normal matrices, we find
\begin{itemize}
	\item{} 
	$P_i$ is a projection; that is $P_i^2-P_i=0$ and $P_i=P_i^*$, $1\leq i\leq m$;  indeed, $P_i$ is the spectral projection for $N$ corresponding to the eigenvalue $\lambda_i$;
	\item{} 
	$P_iP_j=0$, $1\leq i\neq j\leq m$;
	\item{} 
	$\sum_{i=1}^m P_i-I=0$;
\end{itemize}

Therefore,
\begin{itemize}
	\item
	$\|Q_i-Q_i^2\|_u=\|P_i-P_i^2\|_u=0$, so that $Q_i$ is an idempotent, $1\leq i\leq m$;   in fact, the condition above that $\| P_i \|_u = \|Q_i \|_u$ and $\| I_n - P_i\|_i = \| I_n - Q_i \|_u$, combined with Lemma~\ref{projection}, yields that each $Q_i$ is a projection.   
	\item{}
	 $\|Q_iQ_j\|_u=\|P_iP_j\|_u=0$, and thus $Q_iQ_j=0$ $1\leq i\neq j\leq m$;
	\item{}
	$\|\sum_{i=1}^m Q_i-I\|_u=\|\sum_{i=1}^m P_i-I\|_u=0$, and therefore $\sum_{i=1}^m Q_i=I$.
\end{itemize}

Set \[p(z):= z-\sum_{i=1}^m \lambda_i \left( \prod_{1\leq j\leq m, j\neq i}\frac{z-\lambda_j}{\lambda_i-\lambda_j}\right).\]
Then
\[\|A-\sum_{i=1}^m \lambda_i Q_i\|_u =\|p(A)\|_u=\|p(N)\|_u =\|N-\sum_{i=1}^m \lambda_i P_i\|_u =0,\]
and thus
\[A=\sum_{i=1}^m \lambda_i Q_i~~\textup{is~~normal}.\]

The last statement follows trivially from the fact that $\| P_i \|_u = \| Q_i \|_u$ implies that $\mathrm{rank}\, P_i = \mathrm{rank}\ Q_i$ for all $1 \le i \le m$.

\end{proof}

%
%

\begin{rem}
If we replace the hypothesis that $\| \cdot \|_u$ separates projections by rank by the hypothesis that $A$ and $N$ share a common characteristic polynomial, then the fact that $A$ must be unitarily similar to $N$ still follows, as is easy to check.  In the absence of both hypotheses, multiplicity of eigenvalues becomes an issue.   For example, consider the unitarily-invariant norm 
\[
\| T \|_u := s_1(T) + s_2(T), \ \ \ T \in \bM_n(\bC), \]
where as always, $s_1(T) \ge s_2(T) \ge \cdots s_n(T) \ge 0$ are the singular numbers of $T$.  Let $P=\mathrm{diag}\, (1,1,1,0,0)$, $Q=\mathrm{diag}\, (1,1,0,0,0)$ be diagonal projections of ranks 3 and 2 respectively.    For any polynomial $p$ we have that $p(P)=\mathrm{diag}\, (p(1),p(1),p(1),p(0),p(0))$, while $p(Q)=\mathrm{diag}\, (p(1),p(1),p(0),p(0),p(0))$, and so $\|p(P)\|_u=\|p(Q)\|_u$.  Clearly $P$ is not unitarily equivalent to $Q$.
\end{rem}

%
%

%
%

\section{The infinite-dimensional setting:  the operator norm} \label{sec03}

\subsection{}
In~\cite{Lav1936}, Lavrentieff established the following result:

\begin{thm} \label{Lavrentieff}
Let $K \subseteq \bC$ be a compact set, and denote by $\P(K)$ the set of all polynomials with complex coefficients.   The following conditions are equivalent:
\begin{enumerate}
	\item[(a)]
	The family $\P(K)$ is uniformly dense in $\C(K) := \{ f: K \to \bC: f \mbox{ is continuous}\}$. 
	\item[(b)]
	The set $\bC \setminus K$ is connected and the interior $K^o$ of $K$ is empty.
\end{enumerate}
\end{thm}

In view of this result, we refer to a compact set $K$ satisfying condition (b) as a \textbf{Lavrentieff set}.   The relevance of this result for us is that as a consequence of the Spectral Theorem for normal operators (see, e.g.~\cite[Theorem 2.7 (iii) $\iff$ (iv)]{FFM05} ), if $N \in \B(\H)$ is a normal operator, then $N^*$ belongs to the norm-closed unital algebra $\textup{Alg}(N)$ 	generated by $N$ if and only if $\sigma(N)$ is a Lavrentieff set.  In other words, $\textup{Alg}(N) = C^*(N)$ if and only if $\sigma(N)$ is a Lavrentieff set.

%
%

Our first result is general enough to hold in the context of $C^*$-algebras.  

\smallskip

\begin{thm}\label{L set}
Let $\fA, \fB$ be unital $C^*$-algebras, $N\in \fA$ be a normal element with Lavrentieff spectrum and $A \in \fB$.
Suppose that $N$ and $A$ are polynomially isometric.  Then $A$ is normal.
\end{thm}

\begin{proof}
Let $\textup{Alg}(N)$ (resp. $\textup{Alg}(A)$) denote the unital commutative algebra generated by $N$ (resp. by $A$).  
Since $\sigma(N)$ is a Lavrentieff set, $\textup{Alg}(N) = C^*(N)$, as noted above.
Define \[\Phi_0: p(N)\mapsto p(A), \ \ \ \ \ \ p \mbox{ a polynomial}.\]
Then $\Phi_0$ is a unital isometric isomorphism, and can thus be uniquely extended to an isometric isomorphism 
$\Phi:  C^*(N)  \to  \fB$ whose range is $\textup{Alg}(A)$.  It is known, however (see, e.g.~\cite[Proposition~A.5.8]{BLeM2004}) that a contractive homomorphism between $C^*$-algebras is necessarily a ${}^*$-homomorphism, and so $\textup{Alg}(A)$ must also be a $C^*$-algebra, being isometrically isomorphic to $C^*(N)$.   In particular, since $\textup{Alg}(A)$ is abelian and $A^* \in \textup{Alg}(A)$, $A$ must be normal.

\end{proof}

%
%

\begin{rem}
The result of Brooks and Condori~\cite[Corollary 3]{BC2018} is a finite-dimensional version of
Theorem~\ref{L set}, as a finite subset of $\bC$ is obviously a Lavrentieff set.
\end{rem}

%
%

\begin{defn}
Let $K$ be a compact subset of $\bC$, define the polynomially convex hull of $K$ to
be the set $\widehat{K}$ given by
\[\widehat{K}:=\{z\in \bC: |p(z)|\leq  \|p\|_K~~\textup{for}~~\textup{every}~~\textup{polynomial}~~p\},\]
where
\[\|p\|_K:= \max\{|p(z)|:z\in K\}.\]
In fact, $\bC\setminus \widehat{K}$ is exactly the unbounded
component of $\bC\setminus K$ $($see \cite[p206, Proposition 5.3]{Con90}$)$.
\end{defn}

%
%

Recall that given an element $a$ of a Banach algebra $\A$, the \textbf{spectral radius} of $a$ is $\mathrm{spr}(a) := \max \{ |\lambda| : \lambda \in \sigma(a)\}$.

\smallskip

\begin{thm} \label{thm4.03}
Let $N\in B(\H)$ be normal.
The following are equivalent:
\begin{enumerate}
	\item[(a)] 
	$\sigma(N)$ is a Lavrentieff set.
	\item[(b)]
	If $A\in B(\H)$ is polynomially isometric to $N$, then $A$ is normal.
\end{enumerate}
\end{thm}

\begin{proof}

\begin{enumerate}
	\item[(a)] implies (b):  
	This is an immediate consequence of Theorem~\ref{L set}. 
	\item[(b)] implies (a):
	Suppose that  $\sigma(N)$ is not a Lavrentieff set.
	Let $L=\widehat{\sigma(N)}$. Then $L^o\neq \emptyset$. By translating both $A$ and $N$ by the same scalar operator, we may assume without loss of generality that $0\in L^o$.

	Choose $\varepsilon>0$, such that $K:= \overline{B(0, \varepsilon)}\subset L^o$. Denote by $S$ the unilateral shift operator on $l^2$.
	Set $\K=\H\oplus \ell^2$.  Define
	\[T := N\oplus \varepsilon S\in B(\K).\]
	As $S$ is subnormal, given any polynomial $p$, we have that $p(\varepsilon S)$ is subnormal,
	hence by \cite[Proposition 6.10]{Kub11},
	\[\|p(\varepsilon S)\|= \mathrm{spr}(p(\varepsilon S)).\]
	By the Maximal Modulus Principle and the Spectral Mapping Theorem,
	\begin{align*}
	\| p(\varepsilon S)\|
		&=\mathrm{spr} (p(\varepsilon S)) \\
		&=\max\{| p (\lambda)|: \lambda\in \sigma(\varepsilon S)\} \\
		&\leq \max\{| p(\mu)|: \mu\in \sigma(N)\}	\\
		&= \mathrm{spr}(p(N)) \\
		&=\| p(N)\|.
	\end{align*}	
	It follows that $N$ and $T$ are polynomially isometric.
	Clearly, however,  $T$ is not normal.  Since $\H$ is isomorphic to $\H \oplus \ell^2$, we see that $T$ is unitarily equivalent to an $A\in B(\H)$ which must also be non-normal and polynomially isometric to $N$.
\end{enumerate}
\end{proof}

%
%


\section{On the $(c, p)$ norms of Chan, Li and Tu.} \label{sec04}

\subsection{} \label{sec4.01}
We now wish to extend the results of Section~\ref{sec03} to a family of unitarily-invariant norms on $\B(\H)$ first introduced by Chan, Li and Tu in~\cite{CLT2001}.  Before doing so, we first remind the reader of the definition of singular values for not-necessarily compact operators.  

We denote by $\K(\H)$ the ideal of compact operators in $\B(\H)$.    For $K$ in $\K(\H)$, the \textbf{singular numbers} $(s_n(K))_{n=1}^\infty$  of $K$ are defined to be the eigenvalues (repeated according to multiplicity) of $|K| = (K^* K)^{\frac{1}{2}}$.  This notion was  extended to non-compact operators in~\cite{GK1969} as follows.  
Let $T \in \B(\H)$ be an arbitrary operator.   Then $|T| := (T^* T)^{\frac{1}{2}}$ is a positive operator.   Let $\alpha := \max \{ \lambda: \lambda \in \sigma_e(|T|)\}$, where $\sigma_e(X)$ is the \textbf{essential spectrum} of $X \in \B(\H)$, that is, the spectrum of the image $\pi(X)$ of $X \in \B(\H)$ under the canonical map $\pi: \B(\H) \to \B(\H)/\K(\H)$, the Calkin algebra.  It is a consequence of the work of Wolf~\cite{Wol1959} that if $Q:= E_{|T|}( (\alpha, \| T\|])$ denotes the spectral projection of $|T|$ corresponding to the set $(\alpha, \| T\|] \subseteq \bR$, then the spectrum of the compression $(Q |T| Q)_{Q\H}$ of $|T|$ to the range of $Q$ is the closure of a sequence (perhaps finite) of isolated eigenvalues of finite multiplicity of $|T|$ tending to $\alpha$.  Thus, 
\begin{itemize}
	\item{} 
	if $\mathrm{rank}\, Q = 0$, we set $s_n(T) = \alpha$ for all $n \ge 1$.
	\item{}
	If $0 < \mathrm{rank}\, Q = m < \infty$, we denote by $\{ s_n(T)\}_{n=1}^m$ the eigenvalues of $(Q |T| Q)_{Q \H}$, repeated according to multiplicity, and for $n \ge m+1$, we set $s_n(T)= \alpha$. 
	\item{}
	If $\mathrm{rank}\, Q = \infty$, we denote by $\{ s_n(T)\}_{n=1}^\infty$ the eigenvalues of $(Q |T| Q)_{Q \H}$, written in non-increasing order repeated according to their multiplicity.
\end{itemize}
We refer the interested reader to~\cite{GK1969} and to~\cite{Pie1987} for more details regarding the singular numbers of operators.

%
%

\subsection{} \label{sec4.02}
We are now in a position to describe the norms of Chan, Li and Tu.  Let $1 \le p < \infty$, and let $n \ge 1$ be an integer.  Suppose that we are given real numbers $c_1 \ge c_2 \ge \cdots \ge c_n > 0$.   For $T \in \B(\H)$, we define the $(c,p)$-norm of $T$ to be 
\[
\|T\|_{c,p}=(\sum_{j=1}^n c_j s_j(T)^p)^{\frac{1}{p}}.\]
In~\cite{CLT2001}, they proved (amongst other things) that these are indeed unitarily-invariant norms on $\B(\H)$, and that for all $T \in \B(\H)$, 
\[
c_1^{1/p} \| T \| \le \| T \|_{c,p} \le \left( \sum_{j=1}^n c_j \right)^{1/p} \| T \|. \]

This family of norms includes
\begin{itemize}
	\item{}
	the operator norm $\| T \| = s_1(T)$;
	\item{}
	the \textbf{Ky-Fan $n$-norm}  $\| T \|_{[n]} := \sum_{j=1}^n s_j(T)$;
	\item{}
	the \textbf{weighted Ky-Fan $n$-norm} $\| T\|_{[c, n]} := \sum_{j=1}^n c_j s_j(T)$, where $c_1 = 1$ and $c_j \ge c_{j+1}$ for all $1 \le j \le n-1$, as well as 
	\item{}
	the \textbf{$[p,n]$-singular norm}  $\| T \|_{[p,n]} := \left( \sum_{j=1}^n s_j^p(T) \right)^{1/p}$.
\end{itemize}

We also mention the useful fact that if $k \ge n$, then the $(c,p)$-norm defines a unitarily-invariant norm on $\bM_k(\bC)$.

%
%

Our main result in this section is the following.
	
\begin{thm}\label{C_p}
Let $1 \le p < \infty$, $n\in \bN$, $c_1\geq c_2\geq \cdots \geq c_n>0$.  
Suppose that $A$ and $N$ are compact operators, and  that $N$ is normal.  If $A$ and $N$ are polynomially isometric with respect to $\| \cdot \|_{c,p}$, then $A$ is normal.
\end{thm}

\begin{proof}
Let $\{ e_n\}_{n=1}^\infty$ be an orthonormal basis for $\H$.  Set $K:= \sigma(N)\cup\sigma(A)$. Since $N,A$ are compact operators, $K$ is countable, and if it is infinite, it consists of a sequence converging to $0$ (along with $0$ itself).

\smallskip

\noindent{\textbf{Case One.}} $K$ is infinite.

Let us write $K=\{\lambda_i: i\in \bN\}\cup\{0\}$, where $|\lambda_1|\geq |\lambda_2|\geq \cdots>0$, and $\lambda_i\neq \lambda_j$ whenever $i\neq j$, and let us choose $1 \le n_1 < n_2 < \cdots < n_k < \cdots$ such that $|\lambda_{n_k}| > | \lambda_{n_k +1} |$.

For each $k\in \bN$, choose $0 < \varepsilon_k $ such that
\begin{enumerate}
	\item[(i)]
	for  $i \ne j\in \{1,\cdots, n_k\}$,  $B(\lambda_i, 2\varepsilon_k)\cap B(\lambda_j, 2\varepsilon_k)=\varnothing$; and
	\item[(ii)]
	$\{\lambda_j: j > n_k \}\subset B(0, |\lambda_{n_k}|-3\varepsilon_k)$.
\end{enumerate}

Define
\[E_k=(\bigcup_{1\leq i\leq n_k} \overline{B(\lambda_i, \varepsilon_k)})\bigcup \overline{B(0, |\lambda_{n_k}|-2\varepsilon_k)}.\]
Then $E_k$ is a compact set of $\bC$ and $\bC \setminus E_k$ is connected.

Define
\[\Delta_k=\bigcup_{1\leq i\leq n_k} \overline{B(\lambda_i, \varepsilon_k)},\]
and denote by $\chi_{\Delta_k}$ be the characteristic function of $\Delta_k$.  Observe that $\chi_{\Delta_k}$ is holomorphic, hence continuous on $E_k$.
 Let
\[P_k=E_N(\Delta_k),\ \ \ \ (\mbox{resp. } Q_k=E_A(\Delta_k))\]
be the Riesz idempotent for $N$ (resp. for $A$) corresponding to $\Delta_k$.  Since $N$ is normal, $P_k$ is in fact an orthogonal projection.
As $N$ and $A$ are compact, 
\[\rank P_k<\infty \ \ \ \ \ \mbox{ and }\ \ \ \ \ \rank Q_k<\infty.\]

For $k\in \bN$, consider the following function  \[f_k(z)=
\left\{
\begin{array}
    {r@{\quad \quad}l}
    \frac{1}{z},& ~~z\in \bigcup_{1\leq i\leq n_k}B(\lambda_i, 2\varepsilon_k); \\
    0,& ~~z\in  B(0, |\lambda_{n_k}|-3\varepsilon_k).
\end{array}
\right.
\]
Then $f_k$ is also holomorphic on a neighbourhood of $K$.

By Mergelyan's theorem (see~\cite[Theorem 20.5]{Rud87}), we can choose polynomials $t_j^{(k)}$, $j \ge 1$, such that
\[\underset{j\to \infty}{\lim}\|t_j^{(k)}-f_k\|_{E_k}=0.\]
Thus, by the Riesz Functional Calculus~\cite[p201. 4.7(e)]{Con90},
\[\underset{j\to \infty}{\lim}\|t_j^{(k)}(N)-f_k(N)\|=0, \mbox{\ \ \ and \ \ \ }\underset{j\to \infty}{\lim}\|t_j^{(k)}(A)-f_k(A)\|=0,\]
where $\| \cdot \|$ denotes the  operator norm in $\B(\H)$.
By \cite[Proposition 2.4.(a)]{CLT2001},
 \[\underset{j\to \infty}{\lim}\|t_j^{(k)}(N)N-f_k(N)N\|_{c,p}=0, \mbox{\ \ \ and \ \ \ }\underset{j\to \infty}{\lim}\|t_j^{(k)}(A)A-f_k(A)A\|_{c,p}=0,\]
Therefore,
\[\|Q_k\|_{c,p}=\|f_k(A)A\|_{c,p}=\underset{j\to \infty}{\lim}\|t_j^{(k)}(A)A\|_{c,p}=\underset{j\to \infty}{\lim}\|t_j^{(k)}(N)N\|_{c,p}=\|f_k(N)N\|_{c,p}=\|P_k\|_{c,p}.\]

Similarly,
\[\|I-Q_k\|_{c,p}=\|I-P_k\|_{c,p}.\]

Denote $H_k= \ran\, P_k \vee (\ker\, Q_k)^{\perp} \vee \ran\, Q_k \vee \mathrm{span}\{ e_1, e_2, \ldots, e_{\alpha_k} \}$, where $\alpha_k := \mathrm{rank}\, P_k + \mathrm{dim} (\ker\, Q_k)^{\perp} + \mathrm{rank}\, Q_k + 3k $. Set
$R_k$ be the (finite rank) orthogonal projection of $\H$ onto $\H_k$. Clearly $\H_k$ is reducing for both $P_k$ and $Q_k$, and thus $P_k=R_kP_kR_k$, $Q_k=R_kQ_kR_k$.  Moreover, it is readily verified that
\[\|R_k-P_k\|_{c,p}=\|I-P_k\|_{c,p}=\|I-Q_k\|_{c,p}=\|R_k-Q_k\|_{c,p}.\]

Using the fact that the $(c,p)$-norm on $\B(\H_k)$ is unitarily-invariant, it now follows Lemma~\ref{projection} that   $Q_k$ is a projection.
Define
\[N_k:= P_k N = P_kNP_k ,\ \ \ \ \ A_k:= Q_k A = Q_kAQ_k.\]

We now show that $A_k$ and $N_k$ are in fact polynomially isometric.  Let $q$ be an arbitrary polynomial, and write $q = q(0) + q_1$, where $q_1(0) = 0$.   Observe that using~\cite[Proposition 2.4.(a)]{CLT2001} and the fact that $A$ and $N$ are polynomially isometric, we obtain

\begin{eqnarray*}
\|q(N_k)\|_{c,p}&=&\|(q(0)) I+P_kq_1(N)\|_{c,p}\\
&=&\underset{j\to \infty}{\lim}\|q(0)I+t_j^{(k)}(N)Nq_1(N)\|_{c,p}\\
&=&\underset{j\to \infty}{\lim}\|q(0)I+t_j^{(k)}(A)Aq_1(A)\|_{c,p}\\
&=&\|q(0)I+Q_kq_1(A)\|_{c,p}\\
&=&\|q(A_k)\|_{c,p}.\\
\end{eqnarray*}

Notice also that
\[\|q(N_k)\|_{c,p}=\| q(0) R_k+q_1(N_k)\|_{c,p},~~\|q(A_k)\|_{c,p}=\|q(0) R_k+q_1(A_k)\|_{c,p}.\]
We may view each $N_k$ as an operator on $\H_k$, and similarly view $A_k$ as an operator on $\H_k$.
Once again, the fact that the $(c, p)$-norm on $\B(\H_k)$ is unitarily-invariant, combined with Theorem~\ref{thm2.04} implies that $A_k$ is normal, and
\[\sigma(N_k)=\sigma(A_k)=\{\lambda_1, \cdots, \lambda_{n_k}\} \cup \{ 0 \}.\]

Note that $\{Q_k\}_{k=1}^\infty$ is an increasing sequence of projections (of finite rank),
and thus if we define
\[Q=\sup\{Q_k: k\in \bN\}=\textup{SOT}-\underset{k\to \infty}{\lim}Q_k,\]
 then $Q$ is also a projection.
Since $Q_kA=AQ_k$, for all $k\in \bN$, and since $A_k = Q_k Q Q_k$ is normal for all $k$, we see that $AQ=QA$ and that $QAQ$ is normal.

Similarly,  $\{P_k\}_{k=1}^\infty$ is an increasing sequence of projections (of finite rank),
and so
\[P=\sup\{Q_k: k\in \bN\}=\textup{SOT}-\underset{k\to \infty}{\lim}P_k,\]
also defines a projection.  As in the previous case, the fact that  $P_kN=NP_k$ for all $k\in \bN$ implies that $PN=NP$.

Write
\[N=\begin{matrix}\begin{bmatrix}
 PNP&0\\
 0&0
\end{bmatrix}&\begin{matrix}
  \ran P\\
  \ran P^\perp\\
\end{matrix}
\end{matrix},\ \ \ \ \ \ \ \ \ \ A=\begin{matrix}\begin{bmatrix}
QAQ&0\\
 0&Z
\end{bmatrix}&\begin{matrix}
  \ran Q\\
  \ran Q^\perp\\
\end{matrix}
\end{matrix}.\]

If $Z\neq 0$, we can choose $k^*\in \bN$ sufficiently large so that
\[(\sum_{j=1}^n c_j)|\lambda_{n_{k^*}+1}|^p<(\frac{1}{3}\|Z\|)^p.\]

it follows that
\begin{align*}
\|Z\|_{c,p}
	&=\|Q^{\perp}(A-A_{k^*})Q^\perp\|_{c,p}\\
	&\leq\|A-Q_{k^*}A\|_{c,p}\\
	&=\underset{j\to \infty}{\lim}\|A-t_j^{(k^*)}(A)AA\|_{c,p}\\
	&=\underset{j\to \infty}{\lim}\|N-t_j^{(k^*)}(N)NN\|_{c,p}\\
	&=\|N-P_{k^*}N\|_{c,p}\\
	&= \| N - N_{k^*} \|_{(c, p)} \\
	&\leq ((\sum_{j=1}^n c_j)|\lambda_{n_{k^*}+1}|^p)^\frac{1}{p}\\
	&<\frac{1}{3}\|Z\|_{c,p}.
\end{align*}

This obvious contradiction implies, therefore, that $Z=0$. Since $QAQ$ is normal and $Z=0$,  $A$ is normal.

\smallskip

\noindent{\textbf{Case Two.}} $K$ is finite.

The argument here is similar to that of Case One.  The main difference is that instead of considering an infinite sequence of projections $P_k$ and $Q_k$, it suffices to consider a single pair of projections $P$ and $Q$, defined as follows.   

After writing $K = \{ \lambda_1, \lambda_2, \ldots, \lambda_{k_0}\} \cup \{ 0\}$, with $|\lambda_1| \ge | \lambda_2 | \ge \cdots \ge |\lambda_{k_0}| > 0$ and $\lambda_i \ne \lambda_j$ if $i \ne j$, we only choose one value of $\varepsilon > 0$ such that
\begin{enumerate}
	\item[(i)]
	for $i \ne j \in \{ 1, 2, \ldots, k_0\}$, $B(\lambda_i, 2 \varepsilon) \cap B(\lambda_j, 2 \varepsilon) = \varnothing$; and
	\item[(ii)]
	$B(0, 2 \varepsilon) \cap B(\lambda_i, 2 \varepsilon) = \varnothing$, $1 \le i \le k_0$.
\end{enumerate}
We then consider $\Delta := \cup_{1 \le i \le k_0} \overline{B(\lambda_i, \varepsilon)}$.   	Arguing as in Case One, we see that if $P$ and $Q$ are the Riesz projections for $N$ and $A$ corresponding to the set $\Delta$, then $P$ is a projection (as $N$ is normal) and we argue as in Case One that $Q$ is also selfadjoint.    The remainder of the proof follows as above.
\end{proof}

%
%

\begin{rem} \label{rem4.04}
It is clear from our work in Section~\ref{sec03} that this result does not extend directly to general operators on $\B(\H)$, since the result fails for the operator norm, which we have seen to be an example of a $(c,p)$-norm.
\end{rem}


\section{Polynomials without constant term} \label{sec05}

\subsection{} \label{sec5.01}
It is tempting to try to extend the results above to the classes of compact operators lying in the various Schatten $p$-ideals.   Recall that with the singular numbers of a compact operator defined as in Section~\ref{sec04}, for $1 \le p < \infty$, the \textbf{Schatten $p$-class} $\C_p(\H)$ is defined as 
\[
\C_p(\H) := \{ K \in \B(\H): K \mbox{ is compact and }  \sum_{n=1}^\infty (s_n(K))^p < \infty\}, \]
and it is equipped with the unitarily-invariant norm
\[
\| K\|_p := \left( \sum_{n=1}^\infty (s_n(K))^p \right)^{1/p}. \]

Each $\C_p(\H)$ is well-known to be an ideal of $\B(\H)$.  Because of this, the identity operator $I \not \in \C_p(\H)$ for any $1 \le p < \infty$, and so at best, when considering polynomial isometry of two elements $A$ and $B$ of $\C_p(\H)$, we must restrict our attention to those polynomials $q$ which satisfy  $q(0) = 0$.

We now produce an example of two operators $A$ and $N$ in $\C_2(\H)$ (the ideal of Hilbert-Schmidt operators on $\H$) such that $N$ is normal, $A$ and $N$ are polynomially isometric (for polynomials which vanish at $0$), and yet $A$ is not normal.   

%
%

\begin{eg} \label{eg5.02}
Let $\{e_k\}_{k=1}^\infty$ be an orthonormal basis for $\H$, and for $x, y \in \H$, denote by $x \otimes y^*$ the rank-one operator $x\otimes y^*(z) = \langle z, y \rangle x$, $z \in \H$.  
Define $N=e_1\otimes e_1^*+e_2\otimes e_2^*$, $A=e_1\otimes e_1^*+e_1\otimes e_2^*$.  
Then $N$ is a projection and $A$ is an idempotent but it is not a projection.  In particular, $A$ is not normal.
Nevertheless, we find that
\[\|N\|_2=\tr(N^*N)^{\frac{1}{2}}=\sqrt{2},~~\|A\|_2=\tr(A^*A)^{\frac{1}{2}}=\sqrt{2}.\]
Hence for every polynomial $f(z)=a_nz^n+\cdots+a_1z$,
\[\|f(N)\|_2=\|(\sum_{i=1}^n a_i)N\|_2 =\sqrt{2}|\sum_{i=1}^n a_i|,~~\|f(A)\|_2=\|(\sum_{i=1}^n a_i)A\|_2=\sqrt{2}|\sum_{i=1}^n a_i|.\]

That is, $A$ and $N$ are polynomially isometric relative to $\| \cdot \|_2$, considering only polynomials which vanish at $0$.
\end{eg}

%
%

\begin{rem} \label{rem5.025}
It is clear that we did not need to consider the infinite-dimensional Hilbert space setting for the above example.   Indeed, if $\| T\|_2 =
\left ((s_1(T))^2 + (s_2(T))^2 \right)^{1/2}$ denotes the Frobenius norm on $\bM_2(\bC)$, and if $\{ e_1, e_2\}$ is an orthonormal basis for $\bC^2$,  then the operators $N$ and $A$ defined as in the above example again show that Theorem~\ref{thm2.04} fails if we consider only polynomials without constant term.   The problem lies in the fact that in the absence of the identity operator, the Frobenius norm on $\bM_2(\bC)$ (and more generally the $\| \cdot \|_p$ norms on $\bM_n(\bC)$ and on $\C_p(\H)$) are not able to distinguish between projections and idempotents.   This shows that both of the norm conditions of Lemma~\ref{projection} are necessary.

Despite this, there is still something we can say, at least when both $A$ are $N$ are invertible and acting on finite-dimensional spaces.
\end{rem}

%
%

\bigskip

\begin{cor} \label{cor5.03}
Let $A\in \bM_m(\bC)$, $N\in \bM_n(\bC)$ be invertible matrices, and $N$ be normal.  Let  $\kappa := \max(m, n)$, $1 \le p < \infty$, and let $\| \cdot \|_p^{[m]}$ \emph{(}resp. $\| \cdot \|_p^{[n]}$\emph{)} denote the $p$-norm on $\bM_m(\bC)$ \emph{(}resp. the $p$-norm on $\bM_n(\bC)$\emph{)}.
Suppose that $\| r(A) \|_p^{[m]} = \| r(N)\|_p^{[n]}$  for all polynomials  $r$ of degree at most $\kappa$ for which $r(0)= 0$.
Then
\begin{enumerate}
	\item $n=m$;
	\item $A$ is also normal; and
	\item $A$ and $N$ are unitarily similar.
\end{enumerate}
\end{cor}

\begin{proof}
Let $\mu_N, \mu_A$ denote the minimal polynomials of $N$  and of $A$ respectively.   Then $N \mu_N (N) = 0$, and so 
\[
0 = \| N \mu_N(N) \|_p^{[n]} = \|A \mu_N(A) \|_p^{[m]}, \]
whence $A \mu_N(A) = 0$.  Since $A$ is invertible, it follows that $\mu_N(A) = 0$, and so $\mu_A$ divides $\mu_N$.   This part of the argument clearly does not rely upon the fact that $N$ is normal, and so by symmetry, we find that $\mu_N$ divides $\mu_A$ as well.    Without loss of generality, we may therefore assume that $\mu_A = \mu_N$.  We denote this common polynomial by $\mu$.

The fact that both $A$ and $N$ are invertible implies that $\mu(0) \ne 0$.  As such, and keeping in mind that $\mu(A) = \mu(N) = 0$, we set 
\[
\nu(z) = \mu(z) - \mu(0), \ \ \ z \in \bC. \]
Clearly $\nu$ is a polynomial of degree at most $\kappa$ and $\nu(0)= 0$.   By hypothesis, 
\[
\| \mu(0) I_m \|_p^{[m]} = \| \nu(A) \|_p^{[m]} = \| \nu(N)\|_p^{[n]} = \| \mu(0) I_n \|_p^{[n]}. \]
From this it follows that $n=m$, and we simply write $\| \cdot \|_p$ to denote $\| \cdot \|_p^{[n]} = \| \cdot \|_p^{[m]}$.

\bigskip

Let $r(z) = r_0 + r_1 z + r_2 z^2 + \cdots + r_n z^n$ be a polynomial of degree at most $n$.   Set 
\[
q(z) = r_0 \mu(0)^{-1} \nu(z) + (r_1 z + r_2 z^2 + \cdots + r_n z^n), \]
so that $q$ is a polynomial of degree at most $n$ and $q(0) = 0$.   By our hypothesis, 
\begin{align*}
\| r(A) \|_p
	&= \| r_0 I_n + (r_1 A + r_2 A^2 + \cdots r_n A^n ) \|_p \\
	&= \| r_0 (-\mu(0))^{-1} \nu (A) + (r_1 A + r_2 A^2 + \cdots r_n A^n ) \|_p \\
	&= \| q(A) \|_p \\
	&= \| q(N)\|_p \\
	&=  \| r_0 (-\mu(0))^{-1} \nu (N) + (r_1 N + r_2 N^2 + \cdots r_n N^n ) \|_p \\
	&= \| r_0 I_n + (r_1 N + r_2 N^2 + \cdots r_n N^n ) \|_p \\
	&= \| r(N) \|_p.
\end{align*}	
Since $\mu_A = \mu_N$ is a polynomial of degree at most $n$, it follows that if $t$ is any polynomial, then $t(A) = s(A)$ and $t(N) = s(N)$ for some polynomial $s$  of degree at most $n$.  Combining this with the above calculation yields
\[
\| t(A) \|_p = \| t(N) \|_p \ \ \ \mbox{ for all polynomials } t. \]
Since the $p$-norm separates projections by rank, it follows from Theorem~\ref{thm2.04} that $A$ and $N$ are unitarily similar.
\end{proof}


{\bf Acknowledgments.}
This paper was written when the second author was spending his sabbatical leave at the University of Waterloo.  He would like to thank Department of Pure Mathematics of the University of Waterloo, and in particular Professor L.W. Marcoux, for their kind support.



\begin{thebibliography}{99}

\bibitem{BLeM2004}
D.~P. Blecher and C.~Le~Merdy, \emph{Operator algebras and their modules -- an operator space approach}, London Math. Soc. Monographs \textbf{30}, Clarendon Press, Oxford, 2004.

\bibitem{BC2018}
C.~D. Brooks and A.~A. Condori, \emph{A resolvent criterion for
  normality}, Amer. Math. Monthly \textbf{125} (2018), no.~2, 149--156.
    
 \bibitem{CLT2001}
Jor-Ting Chan, Chi-Kwong Li, and Charlies C.~N. Tu, \emph{A class of unitarily
  invariant norms on {$B(H)$}}, Proc. Amer. Math. Soc. \textbf{129} (2001),
  no.~4, 1065--1076. 

\bibitem{Con90}
J.~B. Conway, \emph{A course in functional analysis}, second ed., Graduate
  Texts in Mathematics, vol.~96, Springer-Verlag, New York, 1990.

\bibitem{Dav}
K.~R. Davidson, \emph{Nest algebras}, Pitman Research Notes in Mathematics
  Series, vol. 191, Longman Scientific \& Technical, Harlow, 1988, Triangular
  forms for operator algebras on Hilbert space.

\bibitem{EI1998}
L.~Elsner and Kh.~D. Ikramov, \emph{Normal matrices: an update}, Linear Algebra
  Appl. \textbf{285} (1998), no.~1-3, 291--303.

\bibitem{Fan51}
K.~Fan, \emph{Maximum properties and inequalities for the eigenvalues of
  completely continuous operators}, Proc. Nat. Acad. Sci., U. S. A. \textbf{37}
  (1951), 760--766.

\bibitem{FFM05}
D.~R. Farenick, B.~E. Forrest, and L.~W. Marcoux, \emph{Amenable operators on
  {H}ilbert spaces}, J. Reine Angew. Math. \textbf{582} (2005), 201--228.

\bibitem{FFGSS11}
D.R.~Farenick, V.~Futorny, T.~G.~Gerasimova, V.~V.~Sergeichuk, and N.~Shvai, \emph{A criterion for unitary similarity of
  upper triangular matrices in general position}, Linear Algebra Appl.
  \textbf{435} (2011), no.~6, 1356--1369.

\bibitem{FGS11}
D.R.~Farenick, T.G.~Gerasimova, and N.~Shvai, \emph{A complete
  unitary similarity invariant for unicellular matrices}, Linear Algebra Appl.
  \textbf{435} (2011), no.~2, 409--419.

\bibitem{Ger2012}
T.G.~Gerasimova, \emph{Unitary similarity to a normal matrix}, Linear
  Algebra Appl. \textbf{436} (2012), no.~9, 3777--3783.
  
\bibitem{GK1969}
I.C.~Gohberg and M.G.~Krein, \emph{Introduction to the theory of linear nonselfadjoint operators}, Amer. Math. Soc., Providence, R.I., (1969).

\bibitem{GT1993}
A.~Greenbaum and L.N.~Trefethen, \emph{Do the pseudospectra of a matrix
  determine its behavior?}, Technical Report TR 93-1371, Cornell University,
  1993. ecommons.cornell.edu/bitstream/handle/1813/6145/93-1371.pdf.

\bibitem{GJSW1987}
R.~Grone, C.~R.~Johnson, E.~M.~Sa, and H.~Wolkowicz,
  \emph{Normal matrices}, Linear Algebra Appl. \textbf{87} (1987), 213--225.

\bibitem{Kub11}
C.~S.~Kubrusly, \emph{The elements of operator theory}, second ed.,
  Birkh\"{a}user/Springer, New York, 2011.

\bibitem{Lav1936}
M.A.~Lavrentieff, \emph{Sur les fonctions d'une variable complexe repr\'{e}sentables par les series de polyn\^{o}mes}, 
Acta. Sci. Ind. \textbf{441} (1936).

\bibitem{LiTsi87}
C.K.~Li and N.K~Tsing, \emph{On the unitarily-invariant norms and some
  related results}, Linear and Multilinear Algebra \textbf{20} (1987), no.~2,
  107--119.

\bibitem{Mir60}
L.~Mirsky, \emph{Symmetric gauge functions and unitarily-invariant norms},
  Quart. J. Math. Oxford Ser. (2) \textbf{11} (1960), 50--59.
  
\bibitem{Pie1987}
A.~Pietsch, \emph{Eigenvalues and $S$-numbers}, Cambridge Studies in Advanced Mathematics \textbf{13}, Cambridge Univ. Press, Cambridge, 1987.  

\bibitem{Rin62}
J.~R. Ringrose, \emph{Super-diagonal forms for compact linear operators}, Proc.
  London Math. Soc. (3) \textbf{12} (1962), 367--384.

\bibitem{Rud87}
W.~Rudin, \emph{Real and complex analysis}, third ed., McGraw-Hill Book
  Co., New York, 1987.

\bibitem{VisTre96}
D.~Viswanath and L.N.~Trefethen, \emph{Matrix behaviour, unitary reducibility,
  and {H}adamard products}, Technical Report TR96-1596, Cornell University, 1996.
  ecommons.cornell.edu/handle/1813/7251.
  
\bibitem{Wol1959}
F.~Wolf, \emph{On the essential spectrum of partial differential boundary problems}, Comm. Pure  Appl. Math. \textbf{12} (1959), 211-228.
  
  
\end{thebibliography}

\end{document}